\documentclass[12pt,reqno,oneside]{amsart}
\usepackage{amsmath}
\usepackage[mathcal]{eucal}
\usepackage{verbatim}
\usepackage{hyperref}

\numberwithin{equation}{section}
\newtheorem{theorem}{Theorem}[section]
\newtheorem{corollary}[theorem]{Corollary}
\newtheorem{lemma}[theorem]{Lemma}

\theoremstyle{definition}
\newtheorem{definition}[theorem]{Definition}

\theoremstyle{remark}
\newtheorem{remark}[theorem]{Remark}
\numberwithin{equation}{section}
\newtheorem{example}{Example}

\begin{document}
\title [Coefficient estimate of bi-Bazilevi\u{c} function...]{Coefficient estimate of bi-Bazilevi\u{c} function\\ of complex order based on quasi subordination\\ involving Srivastava-Attiya operator}
\author{G. Murugusundaramoorthy}
\maketitle
\begin{center}
School of Advanced Sciences, VIT University,\\
Vellore 632014, Tamilnadu, India\\
{\tt E-mail:} gmsmoorthy@yahoo.com
\end{center}

\begin{abstract}
In this paper, we introduce and investigate a new subclass of the function class
$\Sigma$ of bi-univalent functions  defined in the open unit disk, which are associated
with the Hurwitz-Lerch zeta function, satisfying subordinate conditions. Furthermore,
we find estimates on the Taylor-Maclaurin coefficients $|a_2|$ and $|a_3|$ for functions
in this new subclass. Several (known or new) consequences of the results are also pointed out.\
\\
{\bf 2010 Mathematics Subject Classification.} Primary 30C45. \
\\
{\bf Key Words and Phrases.} Analytic functions; Univalent functions; Bi-univalent functions;
Bi-starlike and bi-convex functions; Bi-Bazilevi\u{c} functions; Hurwitz-Lerch zeta function;
Jung-Kim-Srivastava integral operator; Libera-Bernardi integral operator.
\end{abstract}

\maketitle

\section{Introduction, Definitions and Preliminaries}
Let $\mathcal{A}$ denote the class of
functions of the form:
\begin{equation}\label{Int-e1}
f(z)=z+\sum\limits_{n=2}^{\infty}a_nz^n,
\end{equation}
which are analytic in the open unit disk
$$\triangle=\{z: z \in
\mathbb{C}\quad {\rm and}\quad |z|<1\}.$$Further, by $\mathcal{S}$
we shall denote the class of all functions in $\mathcal{A}$ which are univalent in $\triangle.$
Some of the important and well-investigated subclasses of the univalent function class $\mathcal{S}$
include (for example) the class $\mathcal{S}^*(\alpha)$ of starlike functions of order $\alpha$
in $\triangle$ and the class $\mathcal{K}(\alpha)$ of convex functions of order $\alpha$ in
$\triangle.$
\par The convolution or Hadamard product of two functions $f,h\in \mathcal{A}$ is denoted by
$f\ast h$ and is defined as
\begin{equation}
(f\ast h)(z)=z+\sum\limits_{n=2}^{\infty }a_{n}b_{n}z^{n},
\end{equation}
where $f(z)$ is given by (\ref{Int-e1}) and $h(z)=z+\sum\limits_{n=2}^{\infty }b_{n}z^{n}.$

\par We recall a general Hurwitz-Lerch Zeta function $\Phi(z, s, a)$ defined by (see \cite{sri4})
\begin{equation}\label{---}
\Phi(z, s, a) := \sum\limits_{k = 0}^{\infty} \frac{z^{k}}{(k + a)^{s}} \quad (a \in \mathbb{C} \setminus
\{   0, -1, -2, \ldots   \}; s \in \mathbb{C}, \mathfrak{R}(s) > 1 \ \mbox{and} \ |z| < 1).
\end{equation}
Several interesting properties and characteristics of the Hurwitz-Lerch zeta function
$\Phi(z, s, a)$ can be found in the recent investigations by Choi and Srivastava \cite{choi},
Ferreira and Lopez \cite{Ferr}, Garg et al \cite{Garg},  Lin et al \cite{Lin2}
and others.

For the class $\mathcal{A},$ Srivastava and Attiya \cite{sri3} (see also Raducanu and Srivastava \cite{Radu}
and Prajapat and Goyal \cite{Praj}) introduced and investigated linear operator:
\begin{equation*}
\mathcal{J}_{\mu}^ b : \mathcal{A}   \longrightarrow   \mathcal{A}
\end{equation*}
defined interms of the Hadamard product ( or convolution ) by
\begin{equation}\label{ch9_1.6}
\mathcal{J}_\mu^b f(z)  = (\mathcal{G}_{\mu, b} \ast f)(z) \quad (z \in \triangle; \ b \in \mathbb{C} \setminus
\{   0, -1, -2, \ldots   \};\, \mu \in \mathbb{C}; \, f \in \mathcal{A} ),
\end{equation}
where, for convenience.
\begin{equation}\label{ch9_1.7}
\mathcal{G}_{\mu}^ b(z)  = (1 + b)^{\mu} [\Phi(z, \mu, b) - b^{-\mu}].
\end{equation}
It is easy to observe from ( given earlier by \cite{Praj}, \cite{Radu}) \eqref{Int-e1}, \eqref{ch9_1.6} and \eqref{ch9_1.7} that
\begin{equation}
\mathcal{J}_\mu^b f(z)  = z + \sum\limits_{k = 2}^{\infty} \Theta_{k} a_{k} z^{k},
\end{equation}
where \begin{equation}\Theta_{k} = \left|\left( \frac{1 + b}{k + b} \right)^{\mu}\right|\end{equation} and ( throughout this paper unless
otherwise mentioned ) the parameters $\mu, b$ are considered as $\mu \in \mathbb{C}$ and
$b \in \mathbb{C} \setminus \{   0, -1, -2, \ldots   \}.$

We note that
\begin{itemize}
  \item For $\mu = 1$ and $b = \nu (\nu > -1)$ generalized Libera-Bernardi integral operator \cite{Reddy}
        \begin{eqnarray*}
          \mathcal{J}_1^ \nu f(z) &=& \frac{1 + \nu}{z^{\nu}}  \int\limits_{0}^{z} t^{\nu - 1} f(t) dt \\
           &=& z + \sum\limits_{k = 2}^{\infty} \left(  \frac{\nu + 1}{k + \nu}  \right) a_{k} z^{k} = \mathcal{L}_{\nu} f(z).
        \end{eqnarray*}

  \item For $\mu = \sigma\, (\sigma > 0)$ and $b = 1,$ Jung-Kim-Srivastava integral operator \cite{Jung}
        \begin{eqnarray*}
          \mathcal{J}_\sigma^ 1 f(z) &=&  \frac{2^{\sigma}}{z \Gamma (\sigma)} \int\limits_{0}^{z}
                                \left( \log \left( \frac{z}{t} \right) \right)^{\sigma - 1} f(t) dt\\
           &=&  z + \sum\limits_{k = 2}^{\infty} \left(  \frac{z}{k + 1}  \right)^{\sigma} a_{k} z^{k} = \mathcal{I}_{\sigma} f(z)
        \end{eqnarray*}
\end{itemize}
closely related to some multiplier transformations studied by Flett \cite{Flet}.

\par
An analytic function $f(z)$ is quasi-subordinate to an
analytic function $h(z),$ in the open unit disk if there exist analytic functions $\phi$ and $w,$ with $w(0)=0$ such
that $|\phi(z)|<1,~~| w(z)|< 1$ and $f(z)=\phi(z)h[w(z)].$ Then we write $f(z) \prec_{\widetilde{q}} h(z).$ If $\phi(z)=1,$ then the
quasi-subordination reduces to the subordination. Also,if $w(z)= z$ then $f(z)= \phi(z)h(z)$ and in this case
we say that $f(z)$ is majorized by $h(z)$ and it is written as $f(z) << h(z)$ in $\triangle.$ Hence it is obvious that
quasi-subordination is the generalization of subordination as well as majorization.It is unfortunate that
the concept quasi-subordination is so for an underlying concept in the area of complex function theory
al though it deserves much attention as it unifies the concept
of both subordination and majorization.
\par Through out this paper it is assumed that $\phi$ is analytic in $\triangle$ with $\phi(0)= 1$ and let
\begin{equation}\label{c7e3}
\phi(z) = 1 + B_{1} z + B_{2} z^2 + B_{3} z^3 + \cdots,\ \ \ (B_{1}> 0).
\end{equation}
also let
\begin{equation}\label{psi}
\psi(z) = C_0 + C_{1} z + C_{2} z^2 + C_{3} z^3 + \cdots, (|\psi(z)<1|) ~~z\in \triangle
\end{equation}

\section{Bi-Univalent function Class $\Sigma$}
It is well known that every function $f\in \mathcal{S}$ has an inverse $f^{-1},$
defined by
\[
f^{-1}(f(z))=z \qquad (z \in \triangle)
\]
and
\[
f(f^{-1}(w))=w \qquad \left(|w| < r_0(f);\,\, r_0(f) \geq \frac{1}{4}\right),
\]
where
\begin{equation}\label{g-e}
g(w)= f^{-1}(w) = w - a_2w^2 + (2a_2^2-a_3)w^3 -
(5a_2^3-5a_2a_3+a_4)w^4+\cdots .
\end{equation}
\par A function $f \in \mathcal{A}$ is said to be bi-univalent in $\triangle$ if both $f(z)$ and
$f^{-1}(z)$ are univalent in $\triangle.$ Let $\Sigma$ denote the class of bi-univalent functions
in $\triangle$ given by (\ref{Int-e1}).

\par Recently there has been triggering interest to study bi-univalent function class $\Sigma$ and
obtained non-sharp coefficient estimates on the first two coefficients $|a_2|$ and $|a_3|$ of
(\ref{Int-e1}). But the coefficient problem for each of the following Taylor-Maclaurin coefficients:
$$|a_n|\qquad (n\in\mathbb{N}\setminus\{1, 2, 3\};\;\;\mathbb{N}:=\{1,2,3,\cdots\}$$
is still an open problem(see\cite{Branna1970,Bran-1979,Bran1985,Lewin,Netany,Taha1981}). Many
researchers (see\cite{BAF-MKA,haya,GMS-TJ,Xu-HMS-AML,HMS-AKM-PG}) have recently
introduced and investigated several interesting subclasses of the bi-univalent function class
$\Sigma$ and they have found non-sharp estimates on the first two Taylor-Maclaurin coefficients
$|a_2|$ and $|a_3|.$

\par Several authors have discussed various subfamilies of Bazilevi\u{c} functions of type $\lambda$
from various perspective. They discussed it from the perspective of convexity, inclusion theorem,
radii of starlikeness and convexity boundary rotational problem, subordination just to mention few.
The most amazing thing is that, it is difficult to see any of this authors discussing the coefficient
inequalities, and coefficient bounds of these subfamilies of Bazilevi\u{c} function most especially
when the parameter $\lambda$ is greater than 1 ($\lambda\in \mathbb{R}$ ). Motivated by the earlier work
of Deniz\cite{DEN} in the present paper we introduce new families of Bazilevi\u{c} functions of complex
order \cite{Noor} of the function class $\Sigma,$ involving Hurwitz-Lerch zeta function, and find
estimates on the coefficients $|a_2| $ and $|a_3|$ for functions in the new subclasses of function
class $\Sigma.$ Several related classes are also considered, and connection to earlier known results
are made.

\begin{definition}
A function $f\in \Sigma $ given by (\ref{Int-e1}) is said to be in the class
$\mathcal{B}^{\mu, b}_{\Sigma}(\gamma, \lambda, \phi)$ if the following conditions are satisfied:
\begin{equation}\label{Defi-1-e1}
 \frac{1}{\gamma}\left(\frac{z^{1-\lambda}(\mathcal{J}_\mu^bf(z))'}{
[\mathcal{J}_\mu^bf(z)]^{1-\lambda}} - 1\right ) \prec_{\widetilde{q}} (\phi(z)-1)
\end{equation}
and
\begin{equation}\label{Defi-1-e2}
 \frac{1}{\gamma} \left(\frac{w^{1-\lambda}(\mathcal{J}_\mu^bg(w))'}{
[\mathcal{J}_\mu^bg(w)]^{1-\lambda}} - 1 \right ) \prec_{\widetilde{q}} (\phi(w)-1)
\end{equation}
where  $\gamma \in \mathbb{C} \setminus \{ 0 \};\lambda \geq 0;z,w \in \triangle$ and the
function $g$ is given by(\ref{g-e}).

\end{definition}
\par On specializing the parameters $\lambda$  one can define the various new subclasses
of $\Sigma$ associated with Hurwitz-Lerch zeta function as illustrated in the following examples.

\begin{example}\label{exam1}
For $\lambda =0$ and a function $f\in \Sigma ,$ given by (\ref{Int-e1}) is said to be in the class
$\mathcal{S}^{\mu, b}_{\Sigma}(\gamma, \phi)$ if the following conditions are satisfied:
\begin{equation}
 \frac{1}{\gamma} \left(\frac{z(\mathcal{J}_\mu^b f(z))'}{
\mathcal{J}_\mu^b f(z)}-1 \right ) \prec_{\widetilde{q}}  (\phi(z)-1)
\end{equation}
and
\begin{equation}
 \frac{1}{\gamma} \left(\frac{w(\mathcal{J}_\mu^b g(w))'}{
\mathcal{J}_\mu^b g(w)}-1 \right ) \prec_{\widetilde{q}}  (\phi(w)-1)
\end{equation}
where $\gamma \in \mathbb{C} \setminus \{ 0 \}; z,w \in \triangle$ and the function $g$ is given
by(\ref{g-e}).
\end{example}

\begin{example}\label{exam2}
For $\lambda = 1$ and a function $f\in \Sigma, $ given by (\ref{Int-e1}) is said to be in the
class $\mathcal{H}^{\mu,b}_{\Sigma}(\gamma,\phi)$ if the following conditions are satisfied:
\begin{equation}
 \frac{1}{\gamma} \left(\mathcal{J}_\mu^b f(z))' - 1 \right) \prec_{\widetilde{q}}  (\phi(z)-1)
\end{equation}
and
\begin{equation}
 \frac{1}{\gamma} \left((\mathcal{J}_\mu^b g(w))' - 1 \right) \prec_{\widetilde{q}}  (\phi(w)-1)
\end{equation}
where $\gamma \in \mathbb{C} \setminus \{ 0 \};z,w \in \triangle$ and the function $g$ is
given by(\ref{g-e}).
\end{example}

\par It is of interest to  note that for  $\gamma = 1$ the class
$\mathcal{B}^{\mu, b}_{\Sigma}(\gamma, \lambda, \phi)$ reduces to the following new subclass
$\mathcal{B}^{\mu, b}_{\Sigma}(\lambda, \phi).$

\begin{definition}
A function $f\in \Sigma $ given by (\ref{Int-e1}) is said to be in the class
$\mathcal{B}^{\mu, b}_{\Sigma}(\lambda, \phi)$ if the following conditions are satisfied:
\begin{equation}\label{Defi-2-e1}
\left(\frac{z^{1-\lambda}(\mathcal{J}_\mu^bf(z))'}{
[\mathcal{J}_\mu^bf(z)]^{1-\lambda}}-1\right)  \prec_{\widetilde{q}}  (\phi(z)-1)
\end{equation}
and
\begin{equation}\label{Defi-2-e2}
\left(\frac{w^{1-\lambda}(\mathcal{J}_\mu^bg(w))'}{
[\mathcal{J}_\mu^bg(w)]^{1-\lambda}}-1\right) \prec_{\widetilde{q}}  (\phi(w)-1)
\end{equation}
where  $\lambda \geq 0;z,w \in \triangle$ and the function $g$ is given by(\ref{g-e}).

\end{definition}

\par For particular values of $\lambda,$ we have

\begin{example}\label{exam5}
For $\lambda =0$ and a function $f\in \Sigma ,$ given by (\ref{Int-e1}) is said to be in the class
$\mathcal{B}^{\mu, b}_{\Sigma}(0, \phi) \equiv \mathcal{S}^{\ast, \mu, b}_{\Sigma}(\phi)$ if the following
conditions are satisfied:
\begin{equation}
\frac{z(\mathcal{J}_\mu^b f(z))'}{\mathcal{J}_\mu^b f(z)}  \prec_{\widetilde{q}}  (\phi(z)-1)
\end{equation}
and
\begin{equation}
\frac{w(\mathcal{J}_\mu^b g(w))'}{\mathcal{J}_\mu^b g(w)}\prec_{\widetilde{q}}  (\phi(w)-1)
\end{equation}
where $z,w \in \triangle$ and the function $g$ is given by(\ref{g-e}).
\end{example}

\begin{example}\label{exam6}
For $\lambda = 1$ and a function $f\in \Sigma, $ given by (\ref{Int-e1}) is said to be in the
class $\mathcal{B}^{\mu,b}_{\Sigma}(1,\phi) \equiv \mathcal{H}^{\mu,b}_{\Sigma}(\phi)$ if the
following conditions are satisfied:
\begin{equation}
(\mathcal{J}_\mu^b f(z))'  \prec_{\widetilde{q}}  (\phi(z)-1)
\end{equation}
and
\begin{equation}
(\mathcal{J}_\mu^b g(w))' \prec_{\widetilde{q}}  (\phi(w)-1)
\end{equation}
where $z,w \in \triangle$ and the function $g$ is given by(\ref{g-e}).
\end{example}

We  now introduce another new subclass of the function class $\Sigma$ of complex order $\gamma \in \mathbb{C}\backslash \{0\}.$

\begin{definition}
A function $f\in \Sigma $ given by (\ref{Int-e1}) is said to be in the
class $\mathcal{G}^{\mu, b}_{\Sigma}(\gamma, \lambda,\phi)$ if the
following conditions are satisfied:
\begin{equation}\label{Defi-1-e1}
\frac{1}{\gamma} \left(
\frac{z (\mathcal{J}_\mu^b f(z))'}{(1-\lambda)\mathcal{J}_\mu^b f(z)+\lambda z (\mathcal{J}_\mu^b f(z))'} -1 \right ) \prec_{\widetilde{q}}  (\phi(z)-1)
\end{equation}
and
\begin{equation}\label{Defi-1-e2}
\frac{1}{\gamma} \left (
\frac{w (\mathcal{J}_\mu^b g(w))'}{(1-\lambda)\mathcal{J}_\mu^b g(w)+\lambda z (\mathcal{J}_\mu^b g(w))'}-1 \right ) \prec_{\widetilde{q}}  (\phi(w)-1)
\end{equation}
where $ \qquad \gamma \in \mathbb{C}\backslash \{0\} ;\,\, 0 \leq \lambda < 1;
\,\, z,w \in \Delta$ and the function $g$ is given by(\ref{g-e}).
\end{definition}
 We note that by taking $\lambda =0$ we have $$\mathcal{G}^{\mu, b}_{\Sigma}(\gamma, 0,\phi) \equiv\mathcal{S}^{\mu, b}_{\Sigma}(\gamma,\phi)$$

\par In the following section we find estimates on the coefficients $|a_2|$ and $|a_3|$ for
functions in the above-defined subclasses $\mathcal{B}^{\mu, b}_{\Sigma}(\gamma,\lambda, \phi)$
of the function class $\Sigma.$

\par In order to derive our main results, we shall need the following lemma:

\begin{lemma}{\rm (see \cite{Pom})}\label{lem-pom}
If $p \in \mathcal{P},$ then $|p_{k}|\leq 2$ for each $k,$ where $\mathcal{P}$ is the family
of all functions $p$ analytic in $\triangle$ for which $\mathfrak{R}\left( p(z) \right) > 0,$
where $p(z) = 1 + p_{1} z + p_{2} z^2 + \cdots$ for $z \in \triangle.$
\end{lemma}

\section{Coefficient Bounds for the Function Class $\mathcal{B}^{\mu,b}_{\Sigma}(\gamma,\lambda, \phi)$}

\par We begin by finding the estimates on the coefficients $|a_2|$ and $|a_3|$ for functions in
the class $\mathcal{B}^{\mu,b}_{\Sigma}(\gamma,\lambda, \phi).$

\begin{theorem}\label{Bi-th1}
Let the function $f(z)$ given by $(\ref{Int-e1})$ be in the class
$\mathcal{B}^{\mu,b}_{\Sigma}(\gamma,\lambda, \phi).$ Then
\begin{equation}\label{bi-th1-b-a2}
|a_2| \leq \frac{|\gamma|~|C_0| B_{1} \sqrt{2 B_{1}}}{ \sqrt{  | \gamma C_0 B_{1}^2  [(\lambda - 1)(\lambda + 2)
\Theta_{2}^2 + 2 (\lambda + 2) \Theta_{3}] -2 (B_{2} - B_{1}) (1 + \lambda)^2 \Theta_{2}^2 | } }
\end{equation}
and
\begin{equation}\label{bi-th1-b-a3}
|a_3| \leq \frac{|\gamma||C_0| B_{1}}{(\lambda+2)\Theta_3}+\frac{|\gamma| |C_1|B_1}{(\lambda+2)\Theta_3}+ \left(\frac{|\gamma| |C_0|B_{1}} {(1+\lambda)\Theta_2}\right)^2.
\end{equation}
\end{theorem}

\begin{proof}
Let $f \in \mathcal{S}^{\mu,b}_{\Sigma}(\gamma,\lambda, \phi)$ and $g = f^{-1}.$
Then there are analytic functions $u, v : \triangle \longrightarrow \triangle$ with $u(0) = 0 = v(0),$
satisfying
\begin{equation}\label{bi-th1-pr-e1}
 \frac{1}{\gamma} \left(\frac{z^{1-\lambda}(\mathcal{J}_\mu^b f(z))'}{
[\mathcal{J}_\mu^b f(z)]^{1-\lambda}}  - 1 \right ) =\psi(z)[ \phi(u(z))-1]
\end{equation}
and
\begin{equation}\label{bi-th1-pr-e2}
 \frac{1}{\gamma} \left(\frac{w^{1-\lambda}(\mathcal{J}_\mu^b g(w))'}{
[\mathcal{J}_\mu^b g(w)]^{1-\lambda}}  - 1 \right )   =\psi(w)[ \phi(u(w))-1].
\end{equation}

Define the functions $p(z)$ and $q(z)$ by
\begin{equation*}
p(z) : = \frac{1 + u(z)}{1 - u(z)} = 1 + p_{1} z + p_{2} z^2 + \cdots
\end{equation*}
and
\begin{equation*}
q(z) : = \frac{1 + v(z)}{1 - v(z)} = 1 + q_{1} z + q_{2} z^2 + \cdots
\end{equation*}
or, equivalently,
\begin{equation}\label{c7e2.5}
u(z) : = \frac{p(z) - 1}{p(z) + 1} = \frac{1}{2}\left[ p_{1} z + \left(p_{2} - \frac{p_{1}^2}{2}\right) z^2 + \cdots \right]
\end{equation}
and
\begin{equation}\label{c7e2.6}
v(z) : = \frac{q(z) - 1}{q(z) + 1} = \frac{1}{2}\left[ q_{1} z + \left(q_{2} - \frac{q_{1}^2}{2}\right) z^2 + \cdots \right].
\end{equation}
Then $p(z)$ and $q(z)$ are analytic in $\triangle$ with $p(0) = 1 = q(0).$ Since $u, v : \triangle \rightarrow \triangle,$
the functions $p(z)$ and $q(z)$ have a positive real part in $\triangle,$ and $|p_{i}| \leq 2$ and $|q_{i}| \leq 2.$
\par Now,
\begin{equation}\label{Exp-p(z)}
\psi(z)[\phi(u(z))-1] =\frac{1}{2}C_0B_1p_1z+ \left[ \frac{1}{2} C_1~B_1p_{1}  + \frac{1}{2}C_0B_1\left(p_{2} - \frac{p_{1}^2}{2}\right)+\frac{C_0B_2}{4}p_1^2\right] z^2 + \cdots
\end{equation}
and
\begin{equation}\label{Exp-q(w)}
\psi(w)[\phi(v(w))-1]=\frac{1}{2}C_0B_1q_1w+ \left[ \frac{1}{2} C_1~B_1q_{1}  + \frac{1}{2}C_0B_1\left(q_{2} - \frac{q_{1}^2}{2}\right)+\frac{C_0B_2}{4}q_1^2\right] w^2 + \cdots
\end{equation}
respectively.

In light of \eqref{Int-e1} - \eqref{c7e3}, from \eqref{Exp-p(z)} and \eqref{Exp-q(w)},
it is evident that
\begin{multline*}
 \frac{(\lambda + 1)}{\gamma} \Theta_{2} a_{2} z + \frac{1}{\gamma} \left[ (\lambda + 2) \Theta_{3} a_{3}
+ \frac{(\lambda - 1)(\lambda + 2)}{2} \Theta_{2}^2 a_{2}^2 \right] z^2 + \cdots\\
= \frac{1}{2}C_0B_1p_1z+ \left[ \frac{1}{2} C_1~B_1p_{1}  + \frac{1}{2}C_0B_1\left(p_{2} - \frac{p_{1}^2}{2}\right)+\frac{C_0B_2}{4}p_1^2\right] z^2 + \cdots
\qquad \qquad \qquad
\end{multline*}
and
\begin{multline*}
 - \frac{(\lambda + 1)}{\gamma} \Theta_{2} a_{2} w + \frac{1}{\gamma} \left[ - (\lambda + 2) \Theta_{3} a_{3}
+ \left(\frac{(\lambda - 1)(\lambda + 2)}{2} \Theta_{2}^2 + 2 (\lambda + 2) \Theta_{3}\right) a_{2}^2 \right]  w^2 + \cdots\\
= \frac{1}{2}C_0B_1q_1w+ \left[ \frac{1}{2} C_1~B_1q_{1}  + \frac{1}{2}C_0B_1\left(q_{2} - \frac{q_{1}^2}{2}\right)+\frac{C_0B_2}{4}q_1^2\right] w^2 + \cdots
\qquad \qquad \qquad
\end{multline*}
which yields the following relations.
\begin{eqnarray}
(1+\lambda)\Theta_2 a_2 & = & \frac{\gamma}{2}C_0 B_{1} p_{1}\label{c7e2.9}\\
\frac{(\lambda-1)(\lambda+2)}{2}\Theta_2^2 a_2^2+(\lambda+2)\Theta_3 a_3 & = &\gamma\left[ \frac{1}{2} C_1~B_1p_{1}  + \frac{1}{2}C_0B_1\left(p_{2} - \frac{p_{1}^2}{2}\right)+\frac{C_0B_2}{4}p_1^2\right]\nonumber\\
\label{c7e2.10}\\
- (\lambda+1)\Theta_2 a_2 & = & \frac{\gamma}{2}C_0 B_{1} q_{1}\label{c7e2.11}
\end{eqnarray}
and
\begin{equation}\label{c7e2.12}
\left(2(\lambda+2)\Theta_3+\frac{(\lambda-1)(\lambda+2)}{2}\Theta_2^2\right) a_2^2 -(\lambda+2)\Theta_3a_3  = \gamma\left[ \frac{1}{2} C_1~B_1q_{1}  + \frac{1}{2}C_0B_1\left(q_{2} - \frac{q_{1}^2}{2}\right)+\frac{C_0B_2}{4}q_1^2\right].
\end{equation}
From \eqref{c7e2.9} and \eqref{c7e2.11}, it follows that
\begin{equation}\label{c7e2.13}
p_{1} = - q_{1}
\end{equation}
and
\begin{equation}\label{c7e2.14}
8(\lambda+1)^2 \Theta_2^2 a_{2}^2 = \gamma^2 C_0^2~B_{1}^2 (p_{1}^2 + q_{1}^2).
\end{equation}
Adding \eqref{c7e2.10} and \eqref{c7e2.12}, we obtain
\begin{equation}\label{th1-a2-star}
\left[(\lambda-1)(\lambda+2)\Theta_2^2 +2(\lambda+2)\Theta_3\right]a_2^2 =
\frac{\gamma C_0 B_{1}}{2} (p_2+q_2) + \frac{\gamma}{4} C_0(B_{2} - B_{1}) (p_{1}^2 + q_{1}^2).
\end{equation}
Using \eqref{c7e2.14} in \eqref{th1-a2-star}, we get
\begin{equation}\label{th1-p1-square}
a_2^2=\frac{\gamma^2 C_0^2~B_{1}^3 (p_2+q_2)}{2 \gamma C_0 B_{1}^2
\left[(\lambda-1)(\lambda+2)\Theta_2^2 +2(\lambda+2)\Theta_3\right] - 4 (B_{2} - B_{1}) (1 + \lambda)^2 \Theta_{2}^2 }.
\end{equation}
Applying Lemma \ref{lem-pom} for the coefficients $p_2$ and $q_2,$ we immediately have
\begin{equation*}
|a_2|^2  \leq  \frac{2|\gamma||C_0|^2 B_{1}^3}{| \gamma C_0 B_{1}^2 \left[(\lambda-1)(\lambda+2)\Theta_2^2 +2(\lambda+2)\Theta_3\right]
- 2 (B_{2} - B_{1}) (1 + \lambda)^2 \Theta_{2}^2 |}.
\end{equation*}
This gives the bound on $|a_2|$ as asserted in \eqref{bi-th1-b-a2}.

\par Next, in order to find the bound on $|a_3|$, by subtracting \eqref{c7e2.12} from
\eqref{c7e2.10}, we get
\begin{multline}\label{th1-a3-cal-e1}
\left[2(\lambda+2)\Theta_3 a_3-2(\lambda+2)\Theta_3 a_2^2 \right]\\ = \frac{\gamma~C_0 B_{1}(p_{2} - q_{2})}{2}
+ \frac{\gamma~C_1 B_{1}(p_1 - q_1)}{2}
+ \frac{\gamma~C_0 (B_{2}-B_1)(p_{1}^2 - q_{1}^2)}{4}.
\end{multline}
Using \eqref{c7e2.13} and \eqref{c7e2.14} in \eqref{th1-a3-cal-e1}, we get
$$
a_3 = \frac{\gamma C_0 B_{1} (p_2-q_2)}{4(\lambda+2)\Theta_3}+\frac{\gamma B_1C_1(p_1-q_1)}{4(\lambda+2)\Theta_3}+\frac{\gamma^2 C_0^2B_{1}^2(p_1^2+q_1^2)}{8(1+\lambda)^2\Theta_2^2}.
$$
Applying Lemma \ref{lem-pom} once again for the coefficients $p_{1}, q_{1}, p_2$ and $q_2,$ we readily get(\ref{bi-th1-b-a3}).
This completes the proof of Theorem \ref{Bi-th1}.
\end{proof}

\par Putting $\lambda = 0$ in Theorem \ref{Bi-th1}, we have the following corollary.

\begin{corollary}\label{sss1}
Let the function $f(z)$ given by $(\ref{Int-e1})$ be in the class $\mathcal{S}^{\mu, b}_{\Sigma}(\gamma, \phi).$
Then
\begin{equation}\label{sss2}
|a_2| \leq \frac{|\gamma|~|C_0| B_{1} \sqrt{B_{1}} }{ \sqrt{ |\gamma~C_0~ B_{1}^2 (2 \Theta_{3} - \Theta_{2}^2) -(B_{2} - B_{1}) \Theta_{2}^2| } }
\end{equation}
and
\begin{equation}\label{sss3}
|a_3| \leq \frac{|\gamma||C_0| B_1}{2\Theta_3} +\frac{|\gamma| |C_1|B_1}{2\Theta_3}+ \left(\frac{|\gamma||C_0| B_{1} } {\Theta_2}\right)^2.
\end{equation}
\end{corollary}

\par Putting $\lambda =1$ in Theorem \ref{Bi-th1}, we have the following corollary.

\begin{corollary}\label{Bi-cor1}
Let the function $f(z)$ given by $(\ref{Int-e1})$ be in the class
$\mathcal{H}^{\mu, b}_{\Sigma}(\gamma, \phi).$ Then
\begin{equation}\label{bi-cor1-b-a2}
|a_2| \leq \frac{|\gamma||C_0| B_{1} \sqrt{B_{1}} }{ \sqrt{ | 3 \gamma~C_0~ B_{1}^2  \Theta_{3} - 4 (B_{2} - B_{1}) \Theta_{2}^2| } }
\end{equation}
and
\begin{equation}\label{bi-cor1-b-a3}
|a_3| \leq \frac{|\gamma||C_0| B_1}{3 \Theta_3}+\frac{|\gamma| |C_1|B_1}{3\Theta_3}+ \left(\frac{|\gamma||C_0| B_1 } {2 \Theta_2}\right)^2 .
\end{equation}
\end{corollary}

If $\mathcal{J}_\mu^b$ is the identity map, from Corollary \ref{sss1} and  \ref{Bi-cor1}, we get the following corollaries.

\begin{corollary}\label{sss1a}
Let the function $f(z)$ given by $(\ref{Int-e1})$ be in the class
$\mathcal{S}^{*}_{\Sigma}(\gamma, \phi).$ Then
\begin{equation}\label{sss2a}
|a_2| \leq \frac{ |\gamma||C_0| B_{1} \sqrt{ B_{1} } }{ \sqrt{ | \gamma~C_0~ B_{1}^2 - (B_{2} - B_{1}) | } }
\end{equation}
and
\begin{equation}\label{sss3a}
|a_3| \leq \frac{|\gamma||C_0| B_1}{2}+\frac{|\gamma| |C_1|B_1}{2}+ \left(|\gamma||C_0| B_1\right)^2.
\end{equation}
\end{corollary}

\begin{corollary}\label{mmm1}
Let the function $f(z)$ given by $(\ref{Int-e1})$ be in the class $\mathcal{H}_{\Sigma}(\gamma, \phi).$
Then
\begin{equation}\label{mmm2}
|a_2| \leq \frac{ |\gamma||C_0| B_{1} \sqrt{ B_{1} } }{ \sqrt{ | 3 \gamma~C_0~ B_{1}^2 - 4 (B_{2} - B_{1}) | } }
\end{equation}
and
\begin{equation}\label{mmm3}
|a_3| \leq \frac{|\gamma||C_0| B_1}{3}+\frac{|\gamma| |C_1|B_1}{3}+ \left(\frac{ |\gamma||C_0| B_1 } {2}\right)^2.
\end{equation}
\end{corollary}

\section{Coefficient Bounds for the Function Class $\mathcal{G}^{\mu,b}_{\Sigma}(\gamma,\lambda, \phi)$}
\begin{theorem}\label{Bi-th2}
Let the function $f(z)$ given by $(\ref{Int-e1})$ be in the class
$\mathcal{G}^{\mu,b}_{\Sigma}(\gamma, \lambda,\phi).$ Then
\begin{equation}\label{bi-th1-b-a2g}
|a_2|\leq\frac{|\gamma|~|C_0| B_1\sqrt{B_1}}{\sqrt{\left|[\gamma~C_0~(\lambda^2-1)B_1^2+ (1-\lambda)^2(B_1-B_2)]\Theta_2^2+ 2\gamma(1-\lambda)C_0B_1^2 \Theta_3\right|}}
\end{equation}
and
\begin{equation}\label{bi-th1-b-a3g}
|a_3| \leq\frac{|\gamma|| C_1| B_1}{2(1-\lambda)\Theta_3}+ \frac{|\gamma|| C_0|~~ B_1}{2(1-\lambda)\Theta_3}+\frac{|\gamma|| C_0|^2~ B_1^2} {(1-\lambda)^2\Theta_2^2} .
\end{equation}
\end{theorem}

\begin{proof}
It follows from (\ref{Defi-1-e1}) and (\ref{Defi-1-e2}) that
\begin{equation}\label{bi-th1-pr-e1}
\frac{1}{\gamma} \left(
\frac{z (\mathcal{J}_\mu^b f(z))'}{(1-\lambda)\mathcal{J}_\mu^b f(z)+\lambda z (\mathcal{J}_\mu^b f(z))'} -1 \right ) =\psi(z)[ \phi(u(z))-1]
\end{equation}
and
\begin{equation}\label{bi-th1-pr-e2}
\frac{1}{\gamma} \left (
\frac{w (\mathcal{J}_\mu^b g(w))'}{(1-\lambda)\mathcal{J}_\mu^b g(w)+\lambda z (\mathcal{J}_\mu^b g(w))'}-1 \right ) = \psi(w)[\phi(v(w))-1],
\end{equation}
where $\psi[\phi(u(z))-1]$ and $\psi[\phi(v(w))-1]$ are given in (\ref{Exp-p(z)}) and (\ref{Exp-q(w)}) respectively.
 \par Now, equating the coefficients in (\ref{bi-th1-pr-e1}) and
(\ref{bi-th1-pr-e2}), we get
\begin{equation}\label{th1-ceof-p1}
\frac{(1-\lambda)}{\gamma}\Theta_2 a_2 = \frac{1}{2} C_0B_{1} p_{1},
\end{equation}
\begin{equation}\label{th1-ceof-p2}
\frac{(\lambda^2-1)}{\gamma}\Theta_2^2 a_2^2+\frac{2(1-\lambda)}{\gamma}\Theta_3 a_3 =
\frac{1}{2} C_1~B_1p_{1}  + \frac{1}{2}C_0B_1\left(p_{2} - \frac{p_{1}^2}{2}\right)+\frac{C_0B_2}{4}p_1^2,
\end{equation}
\begin{equation}\label{th1-ceof-q1}
-\frac{(1-\lambda)}{\gamma}\Theta_2 a_2 = \frac{1}{2}C_0 B_{1} q_{1}
\end{equation}
and
\begin{equation}\label{th1-ceof-q2}
\frac{(\lambda^2-1)}{\gamma}\Theta_2^2 a_2^2 +\frac{2(1-\lambda)}{\gamma}\Theta_3
(2a_2^2-a_3) = \frac{1}{2} C_1~B_1q_{1}  + \frac{1}{2}C_0B_1\left(q_{2} - \frac{q_{1}^2}{2}\right)+\frac{C_0B_2}{4}q_1^2.
\end{equation}
Proceeding as in Theorem \ref{Bi-th1} we get the desired results.

\end{proof}
Note that by taking $\lambda=0$ we get the result as in Corollary \ref{sss1}

\section{Concluding Remark}
\par For the class of strongly starlike functions, the function $\phi$ is given by
\begin{equation}\label{phi01}
\phi(z) = \left( \frac{1 + z}{1 - z} \right)^{\alpha} = 1 + 2 \alpha z + 2 \alpha^2 z^2 + \cdots\quad
(0 < \alpha \leq 1),
\end{equation}
which gives $$B_{1} = 2 \alpha\,\,\,\,\, {\rm and }\,\,\,\, B_{2} = 2 \alpha^2.$$
\par On the other hand For $-1 \leq B \leq A < 1$ if we take
\begin{equation}  \label{ch5-phi03}
\phi(z) = \frac{1 + A z}{1 + Bz} = 1 + (A-B) z -B(A-B) z^2 + B^2(A-B) z^3 +
\cdots.
\end{equation}
then  we have
 $$B_{1} = (A-B), ~B_{2} = -B(A-B).$$
 \par By taking, $A=(1 - 2 \beta)$ where $0 \leq \beta < 1$ ~and $B=-1$ in (\ref{ch5-phi03}), we get,
\begin{eqnarray}
\phi(z) &=& \frac{1 + (1 - 2 \beta) z}{1 - z} \nonumber\\
&=& 1 + 2 (1 - \beta) z + 2 (1 -
\beta) z^2 + 2 (1 - \beta) z^3 + \cdots .\label{ch5-phi02}
\end{eqnarray}
Hence, we have $$B_{1} = B_{2} = 2(1 - \beta).$$
\\Further, by taking $\beta = 0,$ in (\ref{ch5-phi02}), we get,
\begin{equation}  \label{ch5-phi04}
\phi(z) = \frac{1 + z}{1 - z} = 1 + 2 z + 2 z^2 + 2 z^3 + \cdots\quad,
\end{equation}
Hence, $$B_{1} = B_{2} = B_{3} = 2.$$ \par Various Choices of $\phi$  as mentioned above and suitably choosing the values of $B_1$ and $B_2,$ we state some interesting results analogous to Theorem \ref{Bi-th1} , Theorem \ref{bi-th1-b-a2g} and the Corollaries \ref{sss1} to \ref{mmm1}.
\par It is of interest to note that, if $\mu = 1, b = \nu \ (\nu > -1)$ the operator $\mathcal{J}_\mu^b$ turns into
Libera-Bernardi integral operator $\mathcal{L}_{\nu}$ and if $\mu = \sigma (\sigma > 0), \
b = 1$ the operator $\mathcal{J}_\mu^b$ turns into Jung-Kim-Srivastava integral operator
$\mathcal{I}_{\sigma}.$ So, various other interesting corollaries and consequences of
our main results (which are asserted by Theorem \ref{Bi-th1} and \ref{bi-th1-b-a2g} above) can be derived similarly.
If $\mathcal{J}_\mu^b$ is the identity map and $\psi(z)=1$  yields some known results stated in \cite{BAF-MKA,haya,Li-Wang,Xu-HMS-AML,HMS-AKM-PG}.The details involved may be left as an exercise for the interested reader.

\end{document}